\documentclass[reqno]{amsart}
\usepackage[utf8]{inputenc}
\usepackage{amsmath,amsthm,amssymb}
\usepackage{tikz}
\usepackage{tikz-cd}
\usepackage[english]{babel}
\usepackage[
    hmarginratio=1:1,
    textheight=25cm,
    textwidth=15cm
    ]{geometry}

\usepackage{faktor}
\usepackage[hidelinks]{hyperref}

\usepackage{
    wasysym,
    stackengine,
    makebox,
    graphicx,
    mathtools,
    scalerel
}

\renewcommand{\phi}{\varphi}

\DeclareMathOperator*{\bigboxtimes}{\scalerel*{\boxtimes}{\sum}}

\theoremstyle{definition}% normal font

\newtheorem{theorem}{Theorem}
\newtheorem{remark}{Remark}
\newtheorem{definition}{Definition}
\newtheorem{proposition}{Proposition}

\title{A Functorial Generalization of Coxeter Groups}
\author{Vadim Leshkov}
\date{\today}
\address{Vadim Leshkov
\newline\hphantom{iii} Novosibirsk State University
\newline\hphantom{iii} 1 Pirogova St.
\newline\hphantom{iii} 630090, Novosibirsk, Russia}
\email{\href{mailto:v.leshkov@g.nsu.ru}{v.leshkov@g.nsu.ru}}

\begin{document}

\maketitle
\begin{abstract}
In the present work we describe the category $\mathsf{WC}_2$ of weighted 2-complexes and its subcategory $\mathsf{WC}_1$ of weighted graphs.
Since a Coxeter group is defined by its Coxeter graph, the construction of Coxeter groups defines a functor from $\mathsf{WC}_1$ to the category of groups.
We generalize the notion of a Coxeter group by extending the domain of the functor to the category $\mathsf{WC}_2$.
It appears that the resulting functor generalizes the construction of Coxeter groups, Gauss pure braid groups $GV\!P_n$ (introduced by~V.~Bardakov, P.~Bellingeri, and C.~Damiani in~2015), $k$-free braid groups on $n$ strands $G_n^k$ (introduced by V.~Manturov in~2015), and other quotients of Coxeter groups.

\vspace{2mm}
\noindent
\textbf{Key words and phrases:} Coxeter group, Gauss braid group, $k$-free braid group, weighted 2-complex, weighted graph

\vspace{2mm}
\noindent
\textbf{Mathematics Subject Classification 2020:}  20F05, 20J15
\end{abstract}
\vspace*{\fill}
{
\centering
\tableofcontents
}
\vspace*{\fill}
\clearpage
\newpage

\section{Introduction}

The Gauss virtual braid groups $GV\!B_n$ arose during the investigation of Gauss virtual links, which are called in the terminology of V.~Turaev ``homotopy classes of Gauss words'' \cite{Tur2}, and in the terminology of V.~Manturov, ``free knots''\cite{Man1}.

The pure Gauss braid group $GV\!P_n$, first introduced by V.~Bardakov, P.~Bellingeri, and C.~Damiani in their joint work \cite{GVB}, is a subgroup of $GV\!B_n$.
In \cite{GVB} it was shown that $GV\!P_n$ can be presented by the generators $\lambda_{i,j}, 1 \le i < j \le n$ and the relations of the following three types:
\begin{align*}
	\lambda_{i,j}^2&=1,\\
	\lambda_{i,j}\lambda_{k,l} &= \lambda_{k,l}\lambda_{i,j}, \\
	\lambda_{k,i}\lambda_{k,j}\lambda_{i,j} &= \lambda_{i,j}\lambda_{k,j}\lambda_{k,i},
\end{align*}
Here the symbols $i$, $j$, $k$, and $l$ denote pairwise distinct indices. 
From the presentation, it is evident that $GV\!P_n$ is a quotient of a Coxeter group by the following relations
$${(\lambda_{k,i}\lambda_{k,j}\lambda_{i,j})^2 = 1}.$$

In the last ten years, new groups were introduced and studied which are similar to Coxeter groups, but which have relations that contain more than two generators.
It is natural to investigate such groups.
In the present work, we introduce the notion of a weighted 2-complex, which generalizes the notion of a Coxeter graph and the notion of a 2-complex from combinatorial group theory.
Additionally, we introduce the group $G(C)$, the so called generalized Coxeter group, which is constructed from the weighted 2-complex $C$.
The class of generalized Coxeter groups generalizes the class of Coxeter groups in the way that, $G(C)$ is a Coxeter group if $C$ does not contain 2-cells. 

\section{Preliminaries}

For disjoint sets $A$ and $B$, we denote their union $A \cup B$ by $A \sqcup B$.
In this case, we call $A \sqcup B$ the \emph{disjoint union} of the sets $A$ and $B$.
If $A$ and $B$ are arbitrary sets, then by the map (or function)  $f \colon A \to B$ from $A$ to $B$ we indicate a subset of the cartesian product $f \subseteq A \times B$ such that for any $a \in A$ there exists a unique $b \in B$ such that $(a,b) \in f$.
If $(a,b) \in f$ we denote the element $b$ by the symbol $f(a)$ and say that $f(a) = b$.
If $A = \varnothing$, then $A \times B = \varnothing$, and so, $f$ is an empty subset in $A \times B$.
In this case we call the function $f \colon \varnothing \to B$ an \emph{empty function}.
If $C \subseteq A$ is a subset of the set $A$, then we call the map $f|_C \colon C \to B$ defined as $f|_C = f \cap (C \times B)$ the \emph{restriction} of the map $f \colon A \to B$ to the subset $C$.
The element $a \in A$, for which the following identity holds $f(a) = a$, we call a \emph{fixed point} of the map $f$.
We denote the set of all fixed points of a map $f$ as $\text{Fix}(f) := \{ a\in A\, \left|\, f(a) = a \right\}$.
For a set $X$ and a positive integer $n \in \mathbb{N}$ we denote the \emph{diagonal} of the $n$-th cartesian power of the set $X$ by $\text{diag}\left(X^n\right) := \{ (\underbrace{x, x, \ldots, x}_{n\ \text{times}})\, \left|\, x \in X \right\}$.

\section{Category of weighted 2-complexes}

In this chapter we define the category of weighted 2-complexes $\mathsf{WC}_2$.
We build products, coproducts, equalizers, and coequalizers in this category, and we prove that the category is complete and cocomplete. 

\subsection{Main Definitions}

A \emph{graph} or \emph{1-complex} is an ordered tuple of the following form
$$\Gamma = (V \sqcup E, \iota, \alpha, \omega),$$
where $V$ and $E$ are disjoint sets, $\iota \colon V \sqcup E \to V \sqcup E$, $\alpha, \omega \colon V \sqcup E \to V$ are maps which satisfy the following conditions:
$$V = \text{Fix}(\iota),\quad \iota \circ \iota = \text{id}_{V \sqcup E},\quad \alpha \circ \alpha = \alpha,\quad \alpha \circ \iota = \omega.$$
The sets $V$ and $E$ are called the \emph{set of vertices} and the \emph{set of edges} of the graph $\Gamma$, respectively.  
The map $\iota$ is called an \emph{inversion}.
We denote the image of an element $e \in V \sqcup E$ under the map $\iota$ by $e^{-1} := \iota(e)$ and call it the \emph{inverse of (the edge or vertex) $e$}.
The maps $\alpha$ and $\omega$ are called \emph{initial} and \emph{terminal} maps of the graph  $\Gamma$, respectively.
An edge $e \in E$ is called a \emph{loop}, if $\alpha(e) = \omega(e)$.
Edges $e_1, e_2 \in E$ are called \emph{multiple edges}, if $e_1 \neq e_2$ and their initial vertices and their ends coincide, i.e. $\alpha(e_1) = \alpha(e_2)$ and $\omega(e_1) = \omega(e_2)$.

\begin{remark}\label{remark_1_graphs}
From this point on, we assume that graphs do not contain multiple edges or loops. 
\end{remark}

We call a graph $(U, \iota|_{U}, \alpha|_{U}, \omega|_{U})$ a \emph{subgraph} in the graph $\Gamma$, if $U \subseteq V \sqcup E$ is a subset, which is closed under the operations $\iota$, $\alpha$, and $\omega$, i.e. $\iota(U), \alpha(U), \omega(U) \subseteq U$.
For any positive integer $t \in \mathbb{N}$, we call a subgraph $(U, \iota|_{U}, \alpha|_{U}, \omega|_{U})$ in the graph $\Gamma$ a \emph{cycle of length} $t$, if $U = \{v_i, e_i, e_i^{-1}\}_{i = 1}^t$, $\{v_i\}_{i = 1}^t \subseteq V$, $\{e_i, e_i^{-1}\}_{i = 1}^t \subseteq E$ and for every $i \in \{1,\ldots,t-1\}$ $\omega(e_i) = \alpha(e_{i+1}) = v_{i+1}$ and $\omega(e_t) = \alpha(e_1) = v_1$ holds.
We denote a cycle $c$ of length $t$ as $[e_1,\ldots,e_t] = c$, where $e_1,\ldots,e_t$ are all the edges of subgraph $c$, which are listed in the order in which  $\omega(e_i) = \alpha(e_{i+1})$ with $1 \leqslant i < t$ and $\alpha(e_1) = \omega(e_t)$.
We call a subgraph of the form  $c = (\{v\}, \iota|_{\{v\}}, \alpha|_{\{v\}}, \omega|_{\{v\}})$, where $v \in V$, a \emph{cycle of length zero}.
We denote a cycle of length zero $c$ by the symbol $[v] = c$, where $v$ is the single vertex of the subgraph $c$.
We denote the set of all cycles of every non-negative length in the graph $\Gamma$ by $\text{Cyc}(\Gamma)$.
For any $t \in \mathbb{N} \cup \{0\}$ and for any cycle  $c \in \text{Cyc}(\Gamma)$ of length $t$ we denote its length by $\text{len}(c) = t$.

We define a \emph{2-complex} $K$ as an ordered tuple:
$$K = (V \sqcup E \sqcup F, \iota, \alpha, \omega, \partial),$$
where $\Gamma = (V \sqcup E, \iota, \alpha, \omega)$ is a graph, $F$ is a set, called the \emph{set of 2-cells} of the 2-complex $K$, which is disjoint with the sets of vertices and edges: $F \cap (V \sqcup E) = \varnothing$, and $\partial \colon V \sqcup F \to \text{Cyc}(\Gamma)$ is a map which satisfies the following condition:
for every element $f \in V \sqcup F$ the image $\partial(f)$ is a cycle of length zero if and only if $f \in V$.
The map $\partial$ is called the \emph{boundary map} of the 2-complex $K$ and the image $\partial(f)$ of an element $f \in V \sqcup F$ is called a \emph{boundary} of an element $f$.

\begin{remark}\label{remark_2_complexes}
    Further, we assume that in any 2-complex, the boundary of any 2-cell is a cycle of length not less than $3$ and that distinct 2-cell in any 2-complex are required to have distinct boundaries,
    i.e. for 2-cells $f_1, f_2 \in F$ the following implication holds: $\partial(f_1) = \partial(f_2) \Rightarrow f_1 = f_2$.
\end{remark}

A \emph{weighted 2-complex} is a pair $(K, \mu)$, where $K = (V \sqcup E \sqcup F, \iota, \alpha, \omega, \partial)$ is a 2-complex,
$$\mu \colon V \sqcup E \sqcup F \to \mathbb{N} \cup \{\infty\}$$
is a map such that for every $x \in V \sqcup E \sqcup F$, $\mu(x) = 1$ if and only if $x \in V$, and $\mu(x) = \infty$ implies that $x \in E$.
The map $\mu$ is called the \emph{weight function} of the weighted 2-complex $K$.
For a weighted 2-complex $K = (V \sqcup E \sqcup F, \iota, \alpha, \omega, \partial, \mu)$ we denote its corresponding \emph{weighted graph} as $K^{(1)} = (V \sqcup E, \iota, \alpha, \omega, \mu)$.
Further, for any (weighted) 2-complex $K$ we denote $\text{Cyc}(K) = \text{Cyc}(K^{(1)})$.

Let $K_1 = (V_1 \sqcup E_1 \sqcup F_1, \iota_1, \alpha_1, \omega_1, \partial_1)$ and $K_2 = (V_2 \sqcup E_2 \sqcup F_2, \iota_2, \alpha_2, \omega_2, \partial_2)$ be 2-complexes.
\begin{definition}\label{def_1_morph_2_comp}
    A \emph{morphism of 2-complexes} from $K_1$ to $K_2$ is a map
    $$\phi \colon V_1 \sqcup E_1 \sqcup F_1 \to V_2 \sqcup E_2 \sqcup F_2,$$
   which satisfies the following conditions:
    \begin{enumerate}
        \item
        $\phi(V_1) \subseteq V_2$,\ $\phi(E_1) \subseteq V_2 \sqcup E_2$,\ $\phi(F_1) \subseteq F_2 \sqcup V_2$;
        
        \item
        the following diagrams commute:
        \[
        \begin{tikzcd}[row sep=large]
            V_1 \sqcup E_1 \ar[d,"\iota_1"'] \ar[r,"\phi"] & V_2 \sqcup E_2 \ar[d,"\iota_2"] & V_1 \sqcup E_1 \ar[d,"\alpha_1"'] \ar[r,"\phi"] & V_2 \sqcup E_2 \ar[d,"\alpha_2"] & V_1 \sqcup E_1 \ar[d,"\omega_1"'] \ar[r,"\phi"] & V_2 \sqcup E_2 \ar[d,"\omega_2"] \\
            V_1 \sqcup E_1 \ar[r,"\phi"]  & V_2 \sqcup E_2 & V_1 \sqcup E_1 \ar[r,"\phi"]  & V_2 \sqcup E_2 & V_1 \sqcup E_1 \ar[r,"\phi"]  & V_2 \sqcup E_2
        \end{tikzcd}
        \]
        \[
        \begin{tikzcd}[row sep=large]
            V_1 \sqcup F_1 \ar[d,"\partial_1"'] \ar[r,"\phi"] & V_2 \sqcup F_2 \ar[d,"\partial_2"] \\
            \text{Cyc}\left(K_1\right) \ar[r,"\widehat{\phi}"]  & \text{Cyc}\left(K_2\right)
        \end{tikzcd}
        \]
    \end{enumerate}
   where the map
    $$\widehat{\phi} \colon \text{Cyc}\left(K_1\right) \longrightarrow \text{Cyc}\left(K_2\right),\quad [e_1,\ldots,e_t] \mapsto [\phi(e_1),\ldots,\phi(e_t)],$$
    is the map of cycles induced by the map $\phi$.
  It is easy to see that the image of a cycle in $K_1$ under the mapping $\phi$ is a cycle in $K_2$.
  We denote the fact that $\phi$ is a morphism of 2-complexes from $K_1$ to $K_2$ as $\phi \colon K_1 \to K_2$.
\end{definition}

We prove that any morphism of 2-complexes is uniquely defined by its action on vertices.
\begin{proposition}\label{prop_1_morphisms_vertices}
    Let $\phi, \psi \colon K_1 \to K_2$ be morphisms of 2-complexes.
    Then $\phi|_{V_1} = \psi|_{V_1}$ holds if and only if  $\phi = \psi$.
\end{proposition}
\begin{proof}
    Obviously, if  $\phi = \psi$, then $\phi|_{V_1} = \psi|_{V_1}$.
    Let $\phi|_{V_1} = \psi|_{V_1}$ and let there be an edge $e \in E_1$ such that  $\phi(e) \neq \psi(e)$.
    Consider the following two cases: $(\alpha_2 \circ \phi)(e) = (\alpha_2 \circ \psi)(e)$ and $(\alpha_2 \circ \phi)(e) \neq (\alpha_2 \circ \psi)(e)$.
    If $(\alpha_2 \circ \phi)(e) = (\alpha_2 \circ \psi)(e)$, then implementing Remark~\ref{remark_1_graphs} we have $(\omega_2 \circ \phi)(e) \neq (\omega_2 \circ \psi)(e)$ and from the commutativity of the corresponding diagrams, we have:
    $$(\phi \circ \omega_1)(e)= (\omega_2 \circ \phi)(e) \neq (\omega_2 \circ \psi)(e) = (\psi \circ \omega_1)(e).$$
    Thus, there exists a vertex $v = \omega_1(e) \in V_1$ such that $\phi(v) \neq \psi(v)$.
    The existence of such vertex contradicts the assumption that $\phi|_{V_1} = \psi|_{V_1}$.
    If $(\alpha_2 \circ \phi)(e) \neq (\alpha_2 \circ \psi)(e)$, then from the commutativity of the corresponding diagrams:
    $$(\psi \circ \alpha_1)(e) = (\alpha_2 \circ \phi)(e) \neq (\alpha_2 \circ \psi)(e) = (\psi \circ \alpha_1)(e).$$
    There exists a vertex $v = \alpha_1(e) \in V_1$ for which $\phi(v) \neq \psi(v)$.
    The inequality contradicts the assumption that $\phi|_{V_1} = \psi|_{V_1}$.
    In both cases, we reach a contradiction.
    Therefore, $\phi|_{V_1 \sqcup E_1} = \psi|_{V_1 \sqcup E_1}$.
    Let there be a 2-cell $f \in F_1$ such that $\phi(f) \neq \psi(f)$.
    Then, from Remark~\ref{remark_2_complexes}, we can conclude that $(\partial_2 \circ \phi)(f) \neq (\partial_2 \circ \psi)(f)$ and from the commutativity of the corresponding diagrams, we have:
    $$(\widehat{\phi} \circ \partial_1)(f) = (\partial_2 \circ \phi)(f) \neq (\partial_2 \circ \psi)(f) = (\widehat{\psi} \circ \partial_1)(f).$$
    In this way, we have a cycle $c = \partial_1(f) \in \text{Cyc}(K_1)$ such that  $\widehat{\phi}(c) \neq \widehat{\psi}(c)$, where
    $$\widehat{\phi},\, \widehat{\psi} \colon \text{Cyc}(K_1) \to \text{Cyc}(K_2)$$
    are maps of cycles, induced by $\phi$ and $\psi$.
    The obtained inequality contradicts the assumption that  $\phi|_{V_1 \sqcup E_1} = \psi|_{V_1 \sqcup E_1}$.
    Therefore, $\phi = \psi$.
\end{proof}

\begin{remark}\label{remark_3_infinity}
    From this point on, we assume that the infinity symbol $\infty$ is divisible by any natural number and by itself.
\end{remark}
Let $C_1 = (K_1, \mu_1)$ and $C_2 = (K_2, \mu_2)$ be weighted 2-complexes.
\begin{definition}
    A \emph{morphism of weighted 2-complexes} from $C_1$ to $C_2$ is a morphism of 2-complexes $\phi \colon K_1 \to K_2$ such that for any element $x \in V_1 \sqcup E_1 \sqcup F_1$ its weight $\mu_1(x)$ is divisible by the weight of its image $\mu_2(\phi(x))$.
    We denote the fact that $\phi$ is a morphism of weighted 2-complexes from $C_1$ to $C_2$ as $\phi \colon C_1 \to C_2$.
\end{definition}

\begin{proposition}\label{prop_2_composition}
    Let $\phi \colon C_1 \to C_2$, $\psi \colon C_2 \to C_3$ be morphisms of weighted 2-complexes.
    Then, their composition $\psi \circ \phi \colon C_1 \to C_3$ is a morphism of weighted 2-complexes.
\end{proposition}
\begin{proof}
    It is easy to see that the first part of Definition~\ref{def_1_morph_2_comp} is satisfied for the composition.
    Indeed, $(\psi \circ \phi)(V_1) \subseteq \psi(V_2) \subseteq V_3$,\ $(\psi \circ \phi)(E_1) \subseteq \psi(V_2 \sqcup E_2) \subseteq V_3 \sqcup E_3$, $(\psi \circ \phi)(F_1) \subseteq \psi(F_2 \sqcup V_2) \subseteq F_3 \sqcup V_3$.
    Now we prove that the second part of Definition~\ref{def_1_morph_2_comp} is satisfied for the composition. 
    Let $\sigma_i \in \{ \iota_i, \alpha_i, \omega_i \}$, where $i = 1, 2, 3$.
    Consider the following diagram:
    \[
    \begin{tikzcd}[row sep=huge]
            V_1 \sqcup E_1 \ar[d,"\sigma_1"] \ar[r,"\phi"] & V_2 \sqcup E_2 \ar[d,"\sigma_2"] \ar[r,"\psi"] & V_3 \sqcup E_3 \ar[d,"\sigma_3"]\\
            V_1 \sqcup E_1 \ar[r,"\phi"]  & V_2 \sqcup E_2 \ar[r,"\psi"] & V_3 \sqcup E_3
    \end{tikzcd}
    \]
    Due to the fact that the left-hand and the right-hand squares commute, the following identities hold: 
    $$\psi \circ \phi \circ \sigma_1 = \psi \circ \sigma_2 \circ \phi = \sigma_3 \circ \psi \circ \phi.$$
    Now consider the diagram:
    \[
    \begin{tikzcd}[row sep=huge]
        V_1 \sqcup F_1 \ar[d,"\partial_1"] \ar[r,"\phi"] & V_2 \sqcup F_2 \ar[d,"\partial_2"] \ar[r,"\psi"] & V_3 \sqcup F_3 \ar[d,"\partial_3"]\\
        \text{Cyc}(K_1) \ar[r,"\phi"]  & \text{Cyc}(K_2) \ar[r,"\psi"] & \text{Cyc}(K_3)
    \end{tikzcd}
    \]
    Due to the fact that the left-hand and the right-hand squares commute, the following identities hold:
    $$\psi \circ \phi \circ \partial_1 = \psi \circ \partial_2 \circ \phi = \partial_3 \circ \psi \circ \phi.$$
    Therefore, part (2) of Definition~\ref{def_1_morph_2_comp} is satisfied for the composition $\psi \circ \phi$.
    Thus, the composition $\psi \circ \phi \colon K_1 \to K_3$ is a morphism of 2-complexes. 
    Because of the transitivity of the divisibility relation the weight $\mu_1(x)$ is divisible by the weight $\mu_2((\psi \circ \phi)(x))$ for any $x \in V_1 \sqcup E_1 \sqcup F_1$.
    In this way, the composition  $\psi \circ \phi \colon C_1 \to C_3$ is a morphism of weighted 2-complexes. 
\end{proof}

We denote the \emph{category of weighted 2-complexes} by $\mathsf{WC}_\mathsf{2}$.
The objects of $\mathsf{WC}_\mathsf{2}$ form the class $\text{Ob}\left(\mathsf{WC}_\mathsf{2}\right)$ of all possible weighted 2-complexes and for two its objects  $C_1, C_2 \in \text{Ob}\left(\mathsf{WC}_\mathsf{2}\right)$ the set $\text{Hom}_{\mathsf{WC}_\mathsf{2}}(C_1, C_2)$ of all morphisms from $C_1$ to $C_2$ consists of all possible morphisms of weighted 2-complexes from $C_1$ to $C_2$.
The category $\mathsf{WC}_\mathsf{2}$ has the initial and terminal objects denoted by $C_\varnothing$ and $C_\text{pt}$, respectively, defined as:
$$C_\varnothing = (\varnothing, \iota_\varnothing, \alpha_\varnothing, \omega_\varnothing, \partial_\varnothing, \mu_\varnothing),\quad C_\text{pt} = (\{v\}, \iota_\text{pt}, \alpha_\text{pt}, \omega_\text{pt}, \partial_\text{pt}, \mu_\text{pt}),$$
where
\begin{align*}
    \iota_\varnothing, \alpha_\varnothing,  \omega_\varnothing &\colon \varnothing \to \varnothing,\,& \partial_\varnothing &\colon \varnothing \to \varnothing,\,& \mu_\varnothing &\colon \varnothing \to \mathbb{N} \cup \{\infty\},\\
    \iota_\text{pt}, \alpha_\text{pt}, \omega_\text{pt} &\colon \{v\} \to \{v\},\,& \partial_\text{pt} &\colon \varnothing \to \{\, (\{v\}, \iota_\text{pt}, \alpha_\text{pt}, \omega_\text{pt})\, \},\,& \mu_\text{pt} &\colon \varnothing \to \mathbb{N} \cup \{\infty\}.
\end{align*}
The weighted 2-complex $C_\varnothing$ is called the \textit{empty complex}, and the weighted 2-complex $C_\text{pt}$ is called the \emph{single point complex}.
The category $\mathsf{WC}_\mathsf{2}$ contains the subcategory $\mathsf{WC}_\mathsf{1}$, which is called the \emph{category of weighted graphs}.
The class of objects $\text{Ob}\left(\mathsf{WC}_\mathsf{1}\right)$ of the category $\mathsf{WC}_\mathsf{1}$ consists of all weighted graphs (i.e. weighted 2-complexes without any 2-cell), while morphisms in the category $\mathsf{WC}_\mathsf{1}$ are the same as in $\mathsf{WC}_\mathsf{2}$.
The category $\mathsf{WC}_\mathsf{1}$ is a complete subcategory in the category $\mathsf{WC}_\mathsf{2}$, i.e. for any pair  $\Gamma_1, \Gamma_2 \in \text{Ob}\left(\mathsf{WC}_\mathsf{1}\right)$:
$$\text{Hom}_{\mathsf{WC}_\mathsf{1}}(\Gamma_1, \Gamma_2) = \text{Hom}_{\mathsf{WC}_\mathsf{2}}(\Gamma_1, \Gamma_2).$$
Note that the category  $\mathsf{WC}_\mathsf{2}$ is concrete, that is, there exists a \emph{forgetful functor} 
$$\text{V} \colon \mathsf{WC}_\mathsf{2} \longrightarrow \mathsf{Set}$$
from the category $\mathsf{WC}_\mathsf{2}$ of weighted 2-complexes to the category $\mathsf{Set}$ of sets.
For any weighted 2-complex ${C = (V \sqcup E \sqcup F, \iota, \alpha, \omega, \partial, \mu)}$ and for any morphism $\phi \in \text{Hom}_{\mathsf{WC}_\mathsf{2}}(C_1, C_2)$, where $C_1, C_2 \in \text{Ob}(\mathsf{WC}_\mathsf{2})$, the image of the weighted 2-complex $C$ and morphism $\phi$ under the forgetful functor are defined in the following way:
$$\text{V}(C) = V,\quad \text{V}(\phi) = \phi|_{V_1} \colon V_1 \to V_2.$$
The functor $\text{V}$ is faithful, that is, for any $C_1, C_2 \in \text{Ob}\left(\mathsf{WC}_\mathsf{2}\right)$ the map
$$\text{V}_{C_1, C_2} \colon \text{Hom}_{\mathsf{WC}_\mathsf{2}}(C_1, C_2) \longrightarrow \text{Hom}_{\mathsf{Set}}(\text{V}(C_1), \text{V}(C_2)),$$
induced by the functor $\text{V}$, is injective.
We define the \emph{free functor} $\text{FC} \colon \mathsf{Set} \to \mathsf{WC}_\mathsf{2}$.
For a set $M \in \text{Ob}(\mathsf{Set})$ and a map $f \colon A \to B$, where $A, B \in \text{Ob}(\mathsf{Set})$ are sets, we define their images under the functor $\text{FC}$ as:
$$\text{FC}(M) = (M, \text{id}_M, \text{id}_M, \text{id}_M, \partial_\varnothing, \mu_\varnothing),\quad \text{FC}(f) = f \colon \text{FC}(A) \to \text{FC}(B),$$
where $\partial_\varnothing \colon \varnothing \to \varnothing$, $\mu_\varnothing \colon \varnothing \to \mathbb{N} \cup \{\infty\}$.
The functor $\text{FC}$ is faithful and full, that is, for any sets $X, Y \in \text{Ob}(\mathsf{Set})$ there is a bijection:
$$\text{FC}_{X,Y} \colon \text{Hom}_{\mathsf{Set}}(X, Y) \overset{\cong}{\longrightarrow} \text{Hom}_{\mathsf{WC}_\mathsf{2}}(\text{FC}(X), \text{FC}(Y)),$$
induced by the functor $\text{FC}$.
For a set $M$ the weighted 2-complex $\text{FC}(M)$ is called the \emph{free (weighted) 2-complex on} (or \emph{generated by}) the set $M$.

\begin{theorem}
    The functor $\text{FC}$ is left adjoint to functor $\text{V}$, that is for any objects $X \in \text{Ob}(\mathsf{Set})$, $Y \in \text{Ob}(\mathsf{WC}_\mathsf{2})$:
    $$\text{Hom}_{\mathsf{WC}_\mathsf{2}}(\text{FC}(X), Y) \cong \text{Hom}_{\mathsf{Set}}(X, \text{V}(Y)).$$
\end{theorem}
\begin{proof}
    For $X, Y \in \text{Ob}(\mathsf{WC}_2)$ consider the mapping $\Phi_{X,Y}\colon \text{Hom}_{\mathsf{WC}_\mathsf{2}}(\text{FC}(X), Y) \to \text{Hom}_{\mathsf{Set}}(X, \text{V}(Y))$, defined in such a way, that for any $\phi \in \text{Hom}_{\mathsf{WC}_\mathsf{2}}(\text{FC}(X), Y)$:
    $$\Phi_{X,Y}(\phi)\colon X \longrightarrow \text{V}(Y),\quad x \mapsto \phi(x).$$
    Let $\phi, \psi \in \text{Hom}_{\mathsf{WC}_\mathsf{2}}(\text{FC}(X), Y)$ be morphisms of weighted 2-complexes such that $\Phi_{X,Y}(\phi) = \Phi_{X,Y}(\psi)$.
    Then for any $x \in X$, $\phi(x) = \psi(x)$.
    In other words, $\phi$ coincides with $\psi$ on the set of vertices of the complex  $\text{FC}(X)$.
    It means that by Proposition~\ref{prop_1_morphisms_vertices}, we have $\phi = \psi$.
    Hence, $\Phi_{X,Y}$ is injective.
    Let $f \in \text{Hom}_{\mathsf{Set}}(X, \text{V}(Y))$ be an arbitrary mapping between sets.
    The map $f$ is a map from the set of vertices of the complex  $\text{FC}(X)$ to the set of vertices of the complex $Y$.
    It means that $f$ induces a morphism of weighted 2-complexes  $\phi\colon \text{FC}(X) \to Y$, $\phi|_X = f$.
    Thus, $\Phi_{X,Y}(\phi) = f$.
    It follows that $\Phi_{X,Y}$ is surjective.
    Therefore, the mapping $\Phi_{X,Y}$ is bijective.
    Which means that $\Phi_{X,Y}$ defines an isomorphism $\text{Hom}_{\mathsf{WC}_\mathsf{2}}(\text{FC}(X), Y) \cong \text{Hom}_{\mathsf{Set}}(X, \text{V}(Y))$.
    It is easy to see that the class $\Phi = \{ \Phi_{X,Y} \}_{X,Y \in \text{Ob}(\mathsf{WC}_2)}$ defines a natural equivalence of bifunctors
    $$\Phi\colon \text{Hom}_{\mathsf{WC}_\mathsf{2}}(\text{FC}(-), -)\ \Longrightarrow\ \text{Hom}_{\mathsf{Set}}(-, \text{V}(-)).$$
\end{proof}

\subsection{Operations on weighted 2-complexes}
Let $K = (V \sqcup E \sqcup F, \iota, \alpha, \omega, \partial)$ be a 2-complex and let $\sim$ be an equivalence relation on the set of vertices $V$.
Any such equivalence relation $\sim$ on the set of vertices $V$ can be extended to the set $V \sqcup E \sqcup F$ in the following way.
For all $e_1, e_2 \in E$, $f_1, f_2 \in F$ assume that:
\begin{itemize}
    \item[(1)] $e_1 \sim e_2$ if and only if $\alpha(e_1) \sim \alpha(e_2)$ and $\omega(e_1) \sim \omega(e_2)$;
    \item[(2)] $f_1 \sim f_2$ if and only if $\text{len}(f_1) = \text{len}(f_2)$ and $f_1 = [e_1, \ldots,e_t]$, $f_2 = [e_1', \ldots, e_t']$, where $e_i \sim e_i'$ for all $1 \leqslant i \leqslant t$.
\end{itemize}
We denote the equivalence class of an element $x \in V \sqcup E \sqcup F$ as $[x]_\sim$\, .

\begin{definition}\label{def_3_eq_rel}
    Any equivalence relation which is extended from the set of vertices $V$ to the 2-complex $K$ as described above is called an \emph{equivalence relation on the 2-complex $K$}
\end{definition}

Let $\sim$ be an equivalence relation on $K$.

\begin{definition}
    The \emph{quotient complex of 2-complex $K$ by the relation $\sim$} is the 2-complex
    $$\faktor{K}{\sim}\ = \left(\, (V \sqcup E \sqcup F)/\sim\, ,\ \widetilde{\iota},\ \widetilde{\alpha},\ \widetilde{\omega},\ \widetilde{\partial}\, \right),$$
    where $\widetilde{\iota},\ \widetilde{\alpha},\ \widetilde{\omega} \colon \faktor{(V \sqcup E)}{\sim} \longrightarrow \faktor{(V \sqcup E)}{\sim}$,\quad $\widetilde{\partial} \colon \faktor{(V \sqcup F)}{\sim} \longrightarrow \text{Cyc}\left( \faktor{K}{\sim} \right)$,
    $$\widetilde{\iota} \colon \left[ x \right]_{\sim} \mapsto \left[\iota(x)\right]_{\sim},\quad \widetilde{\alpha} \colon \left[ x \right]_{\sim} \mapsto \left[\alpha(x)\right]_{\sim},\quad \widetilde{\omega} \colon \left[ x \right]_{\sim} \mapsto \left[\omega(x)\right]_{\sim},\quad \widetilde{\partial} \colon \left[ x \right]_{\sim} \mapsto \faktor{\partial(x)}{\sim}.$$
    Let $C = (K, \mu)$ be a weighted 2-complex.
    The \emph{quotient complex of the weighted 2-complex $C$ by the relation $\sim$} is the weighted 2-complex $\faktor{C}{\sim}\ = \left(\faktor{K}{\sim},\ \widetilde{\mu}\right)$, where
    \begin{align*}
        \widetilde{\mu} \colon \faktor{(V \sqcup E \sqcup F)}{\sim} &\to \mathbb{N} \cup \{\infty\},\\
        \left[x\right]_{\sim} &\mapsto \text{GCD}\left\{\mu(s)\, |\ s \in V \sqcup E \sqcup F ,\ s \sim x \right\}.
    \end{align*}
\end{definition}

\begin{proposition}
    The maps $\widetilde{\iota}$, $\widetilde{\alpha}$, $\widetilde{\omega}$,
    $\widetilde{\partial}$, $\widetilde{\mu}$ are well-defined.
\end{proposition}

We show that the category $\mathsf{WC}_\mathsf{2}$ has equalisers of all pairs of parallel morphisms.
Consider the following diagram in the category $\mathsf{WC}_\mathsf{2}$:
\begin{equation}\label{diagram_1_parallel_pair}
    \begin{tikzcd}[row sep=huge]
        C_1 \arrow[r, shift left, "\phi"] \arrow[r, shift right, "\psi"'] & C_2
    \end{tikzcd},
\end{equation}
where $C_1, C_2 \in \text{Ob}(\mathsf{WC}_\mathsf{2})$, $C_1 = (V_1 \sqcup E_1 \sqcup F_1, \iota_1, \alpha_1, \omega_1, \partial_1, \mu_1)$, $C_2 = (V_2 \sqcup E_2 \sqcup F_2, \iota_2, \alpha_2, \omega_2, \partial_2, \mu_2)$.
Consider the set $V_\text{eq} \sqcup E_\text{eq} \sqcup F_\text{eq} \subseteq V_1 \sqcup E_1 \sqcup F_1$, where:
$$V_\text{eq} = \{ v \in V_1\, |\, \phi(v) = \psi(v) \},\quad E_\text{eq} = \{ e \in E_1\, |\, \phi(e) = \psi(e) \},\quad F_\text{eq} = \{ f \in F_1\, |\, \phi(f) = \psi(f) \}.$$
\begin{definition}
    \emph{An equalizer of a pair of morphisms $\phi$ and $\psi$} is the weighted 2-complex of the following form: $$\text{Eq}(\phi,\psi) = \left( V_\text{eq} \sqcup E_\text{eq} \sqcup F_\text{eq},\ \iota|_{V_\text{eq} \sqcup E_\text{eq}},\ \alpha|_{V_\text{eq} \sqcup E_\text{eq}},\ \omega|_{V_\text{eq} \sqcup E_\text{eq}},\ \partial|_{V_\text{eq} \sqcup F_\text{eq}},\ \mu|_{E_\text{eq}} \right)$$
    together with the natural embedding $\text{eq}\colon \text{Eq}(\phi,\psi) \xhookrightarrow{} C_1$, $x \mapsto x$.
\end{definition}

\begin{proposition}[Universal property of the equaliser]
    The embedding $\text{eq}$ of an equalizer of a pair of morphisms $\phi$ and $\psi$ to a weighted 2-complex $C_1$ has the following properties:
    \begin{enumerate}
        \item
        the following diagram commutes:
        $$\begin{tikzcd}[row sep=huge]
            \text{Eq}(\phi, \psi) \ar[r,hook, "\text{eq}"] & C_1 \arrow[r, shift left, "\phi"] \arrow[r, shift right, "\psi"'] & C_2
        \end{tikzcd}$$
        that is $\phi \circ \text{eq} = \psi \circ \text{eq}$;
    
        \item 
        for any morphism of weighted 2-complexes $\sigma \colon X \to C_1$, $X \in \text{Ob}(\mathsf{WC}_\mathsf{2})$, if $\sigma$ satisfies $\phi \circ \sigma = \psi \circ \sigma$, then $\sigma$ factors through the morphism $\text{eq}$ uniquely, i.e. there exists a unique morphism of weighted 2-complexes $\rho \colon X \to \text{Eq}(\phi, \psi)$, which makes the following diagram commute:
        \[
        \begin{tikzcd}
            X \ar[d,dashed,"\rho"'] \ar[rd,"\sigma"] & & \\
            \text{Eq}(\phi, \psi) \ar[r,hook, "\text{eq}"'] & C_1 \ar[r, shift left, "\phi"] \ar[r, shift right, "\psi"'] & C_2
        \end{tikzcd}
        \]
    \end{enumerate}
\end{proposition}
\begin{proof}
    Property (1) is satisfied by the definition of the morphism $\text{eq}$.
    Now we prove property (2).
    Let $X = (V_X \sqcup E_X \sqcup F_X, \iota_X, \alpha_X, \omega_X, \partial_X, \mu_X)$ be a weighted 2-complex and let $\sigma \colon X \to C_1$ be a morphism of weighted 2-complexes such that $\phi \circ \sigma = \psi \circ \sigma$.
    Define the morphism $\rho\colon X\to\text{Eq}(\phi,\psi)$.
    For any vertex $v \in V_X$ of a weighted 2-complex $X$, assume that $\rho(v) = \sigma(v)$.
    By specifying the values of $\rho$ on the vertices we define a morphism of weighted 2-complexes $\rho$.
    Note that $\text{eq}|_{V_\text{eq}} \circ \rho|_{V_X} = \sigma|_{V_X}$.
    Thus, $\text{eq} \circ \rho = \sigma$.
    At the same time, if $\rho' \colon X \to \text{Eq}(\phi,\psi)$ is a morphism which satisfies $\text{eq} \circ \rho' = \sigma$, then $\text{eq} \circ \rho' = \text{eq} \circ \rho$.
    Therefore, due to the fact that $\text{eq}$ is an embedding (a monomorphism), $\rho' = \rho$.
    It means that the morphism $\rho$ is unique.
\end{proof}

Define the relation $\sim$ on the set of vertices $V_2$ of the complex $C_2$ in the diagram \eqref{diagram_1_parallel_pair}.
For a pair of vertices $u, v \in V_2$ assume that $u \sim v$ if and only if there exists an element $x \in V_1 \sqcup E_1 \sqcup F_1$, for which $\phi(x) = u$ and $\psi(x) = v$.
The described relation $\sim$ induces an equivalence relation which satisfies Definition~\ref{def_3_eq_rel}, i.e. $\sim_{\phi,\psi}$ is an equivalence relation on the weighted 2-complex $C_2$.

\begin{definition}
    A \emph{coequalizer of a pair of morphisms $\phi$ and $\psi$} is the following quotient complex: $$\text{Coeq}(\phi,\psi) := \faktor{C_2}{\sim_{\phi,\psi}}$$
    together with the natural projection $\text{coeq} \colon C_2 \twoheadrightarrow \text{Coeq}(\phi,\psi)$, $x \mapsto [x]_{\sim_{\phi,\psi}}$.
\end{definition}

\begin{proposition}[Universal property of the coequalizer]
    The morphism $\text{coeq}$ from the weighted 2-complex $C_2$ to the coequalizer of a pair of morphisms $\phi$ and $\psi$ has the following properties:
    \begin{enumerate}
        \item
        the following diagram commutes:
        $$\begin{tikzcd}[row sep=huge]
            C_1 \arrow[r, shift left, "\phi"] \arrow[r, shift right, "\psi"'] & C_2 \ar[r,->>,near start,"\text{coeq}"] & \text{Coeq}(\phi,\psi)
        \end{tikzcd}$$
        that is $\text{coeq} \circ \phi = \text{coeq} \circ \psi$;
    
        \item 
        for any morphism of weighted 2-complexes $\sigma \colon C_2 \to X$, $X \in \text{Ob}(\mathsf{WC}_\mathsf{2})$, if $\sigma$ satisfies  $\sigma \circ \phi = \sigma \circ \psi$, then  $\text{coeq}$ factors through $\sigma$ uniquely, i.e. there exists a unique morphism of weighted 2-complexes $\rho \colon \text{Coeq}(\phi, \psi) \to X$, which makes the following diagram commute:
        $$\begin{tikzcd}
            C_1 \arrow[r, shift left, "\phi"] \arrow[r, shift right, "\psi"'] & C_2 \ar[rd,"\sigma"'] \ar[r,->>,near start,"\text{coeq}"] & \text{Coeq}(\phi,\psi) \ar[d,dashed,"\rho"]\\
            & & X
        \end{tikzcd}$$
    \end{enumerate}
\end{proposition}
\begin{proof}
    First we prove property (1).
    Let $x \in V_1 \sqcup E_1 \sqcup F_1$.
    Then $\phi(x) \sim_{\phi,\psi} \psi(x)$ by the definition of relation $\sim_{\phi,\psi}$.
    Thus, $[\phi(x)] = [\psi(x)]$ and $\text{coeq} \circ \phi = \text{coeq} \circ \psi$.

    Now we prove property (2).
    Let $X = (V_X \sqcup E_X \sqcup F_X, \iota_X, \alpha_X, \omega_X, \partial_X, \mu_X)$ be a weighted 2-complex and let $\sigma \colon C_2 \to X$ be a morphism of weighted 2-complexes such that $\sigma \circ \phi = \sigma \circ \psi$.
    Define the morphism $\rho\colon \text{Coeq}(\phi,\psi) \to X$.
    For any equivalence class $[v]_{\sim_{\phi,\psi}}$ of a vertex $v \in V_2$ assume that $\rho\left([v]_{\sim_{\phi,\psi}}\right) = \sigma(v)$.
    The value $\rho\left([v]_{\sim_{\phi,\psi}}\right)$ does not depend on the choice of a representative in the class $[v]_{\sim_{\phi,\psi}}$ because the pair $\phi$, $\psi$ is coequalized by $\sigma$.
    By specifying the values of $\rho$ on the vertices we define a morphism of weighted 2-complexes $\rho$.
    Note that $\rho|_{V_X} \circ \text{coeq}|_{V_2} = \sigma|_{V_2}$.
    It means that $\rho \circ \text{coeq} = \sigma$.
    At the same time, if $\rho' \colon \text{Coeq}(\phi,\psi) \to X$ is a morphism which satisfies $\rho' \circ \text{coeq} = \sigma$, then $\rho' \circ \text{coeq} = \rho \circ \text{coeq}$.
    Therefore, due to the fact that $\text{coeq}$ is a projection (an epimorphism), $\rho' = \rho$.
    It means that the morphism $\rho$ is unique.
\end{proof}

Let $\{ C_i \}_{i \in I}$ be a set of weighted 2-complexes, $C_i = (V_i \sqcup E_i \sqcup F_i, \iota_i, \alpha_i, \omega_i, \partial_i, \mu_i)$ for all $i \in I$.

\begin{definition}
    The \emph{disjoint union} of the set of weighted 2-complexes $\{ C_i \}_{i \in I}$ is the weighted 2-complex
    $$\bigsqcup_{i \in I} C_i := \left( \bigsqcup_{i \in I} \left(V_i \sqcup E_i \sqcup F_i\right),\, \iota_\sqcup,\, \alpha_\sqcup,\, \omega_\sqcup,\, \partial_\sqcup,\, \mu_\sqcup \right),$$
    where for all $i \in I$, $e \in V_i \sqcup E_i$, $f \in V_i \sqcup F_i$, $x \in V_i \sqcup E_i \sqcup F_i$:
    $$\iota_\sqcup(e) = \iota_i(e),\quad \alpha_\sqcup(e) = \alpha_i(e),\quad \omega_\sqcup(e) = \omega_i(e),\quad \partial_\sqcup(f) = \partial_i(f),\quad \mu_\sqcup(x) = \mu_i(x).$$
\end{definition}

For any collection of weighted 2-complexes  $\{ C_i \}_{i \in I}$ and for any $j \in I$ there exists a natural embedding $\eta_j \colon C_j \xhookrightarrow{} \bigsqcup_{i \in I} C_i$ of $j$-th complex $C_j$ to the disjoint union $\bigsqcup_{i \in I} C_i$.

\begin{proposition}[Universal property of the disjoint union]
    The disjoint union defines the coproduct in the category $\mathsf{WC}_\mathsf{2}$, i.e. for a set of morphisms $\{\phi_i\colon C_i \to X\}_{i \in I}$, $X \in \text{Ob}(\mathsf{WC}_\mathsf{2})$, there exists a unique morphism $\phi \colon \bigsqcup_{i \in I} C_i \to X$, which for every $j \in I$ makes the following diagram commute:
    \[
    \begin{tikzcd}
        C_j \ar[dr,"\phi_j"'] \ar[r,"\eta_j"] & \bigsqcup_{i \in I} C_i \ar[d,dashed,"\phi"]\\
        & X
    \end{tikzcd}
    \]
    We denote the morphism $\phi$ as $\bigsqcup_{i \in I} \phi_i$ and we call it the \emph{disjoint union (or coproduct) of a family of morphisms} $\{\phi_i\}_{i \in I}$.
\end{proposition}
\begin{proof}
    Define the morphism $\phi \colon \bigsqcup_{i \in I} C_i \to X$.
   For each $j \in I$ and for every $x \in V_j \sqcup E_j \sqcup F_j$ assume $\phi(x) = \phi_j(x)$.
    In other words, $\phi = \bigsqcup_{i \in I} \phi_i$.
    It is clear that $\phi$ is a morphism of weighted 2-complexes and that $\phi$  is the unique morphism which satisfies the identities $\phi_i = \phi \circ \eta_i$,\ $i \in I$.
\end{proof}

The following construction generalizes the notion of the \emph{strong product of graphs} introduced by G.~Sabidussi in 1960 \cite{Sab}.

\begin{definition}
    The \emph{strong product} of the set of weighted 2-complexes $\{ C_i \}_{i \in I}$ is the weighted 2-complex
    $$\bigboxtimes_{i \in I} C_i := \left( \prod_{i \in I} \left(V_i \sqcup E_i\right) \cup \prod_{i \in I} \left(V_i \sqcup F_i\right),\, \iota_\boxtimes,\, \alpha_\boxtimes,\, \omega_\boxtimes,\, \partial_\boxtimes,\, \mu_\boxtimes \right),$$
    where for any $x = (x_i)_{i \in I} \in \prod_{i \in I} \left(V_i \sqcup E_i\right) \cup \prod_{i \in I} \left(V_i \sqcup F_i\right)$
    $$\iota_\boxtimes(x) = (\iota_i(x_i))_{i \in I},\quad \alpha_\boxtimes(x) = (\alpha_i(x_i))_{i \in I},\quad \omega_\boxtimes(x) = (\omega_i(x_i))_{i \in I},\quad \mu_\boxtimes(x) = \text{LCM}\{\mu_i(x_i)\, |\, i \in I\}$$
    and for any 2-cell $f = (f_i)_{i \in I} \in \prod_{i \in I} (V_i \sqcup F_i)$
    $$\partial_\boxtimes(f) = \bigboxtimes_{i \in I} \partial_i(f_i).$$
\end{definition}

For any collection of weighted 2-complexes $\{ C_i \}_{i \in I}$ and for any $j \in I$ there exists a natural projection $\pi_j \colon \bigboxtimes_{i \in I} C_i \twoheadrightarrow C_j$ of the strong product $\bigboxtimes_{i \in I} C_i$ onto the $j$-th complex $C_j$.

\begin{proposition}[Universal property of the strong product]
    The strong product defines the product in the category $\mathsf{WC}_\mathsf{2}$, i.e. for a set of morphisms $\{\phi_i\colon X \to C_i\}_{i \in I}$, $X \in \text{Ob}(\mathsf{WC}_\mathsf{2})$, there exists a unique morphism $\phi \colon X \to \bigboxtimes_{i \in I} C_i$, which for every  $j \in I$ makes the following diagram commute:
    \[
    \begin{tikzcd}
        \bigboxtimes_{i \in I} C_i \ar[r,->>,"\pi_j"] & C_j\\
        X \ar[u,dashed,"\phi"] \ar[ur,"\phi_j"'] &
    \end{tikzcd}
    \]
    We denote the morphism $\phi$ as $\bigboxtimes_{i \in I} \phi_i$ and we call it the strong product of a family of morphisms $\{\phi_i\}_{i \in I}$.
\end{proposition}
\begin{proof}
    Let $X = (V \sqcup E \sqcup F, \iota, \alpha, \omega, \partial, \mu)$ be a weighted 2-complex.
    Define the morphism $\phi \colon X \to \bigboxtimes_{i \in I} C_i$.
    For each $x \in V \sqcup E \sqcup F$ assume that $\phi(x) = (\phi_i(x))_{i \in I}$.
    Then for any $j \in I$, $\phi_j = \pi_j \circ \phi$ is satisfied.
    It is clear, that any morphism which satisfies this identity coincides with $\phi$.
\end{proof}

The following properties of operations of the disjoint union and strong product are obvious.
\begin{proposition}
    For any weighted 2-complexes $C, C_1, C_2, C_3 \in \text{Ob}(\mathsf{WC}_\mathsf{2})$ the following identities hold:
    \begin{align*}
        C \sqcup C_\varnothing =\ &C = C_\varnothing \sqcup C, \,& C_1 \sqcup C_2 &= C_2 \sqcup C_1, \,& C_1 \sqcup (C_2 \sqcup C_3) &= (C_1 \sqcup C_2) \sqcup C_3;\\
        C \boxtimes C_\text{pt} \cong\ &C \cong C_\text{pt} \boxtimes C, \,& C_1 \boxtimes C_2 &\cong C_2 \boxtimes C_1, \,& C_1 \boxtimes (C_2 \boxtimes C_3) &\cong (C_1 \boxtimes C_2) \boxtimes C_3.
    \end{align*}
\end{proposition}

\begin{theorem}
    The category $\mathsf{WC}_\mathsf{2}$ is complete and cocomplete, that is in the category $\mathsf{WC}_\mathsf{2}$ limits and colimits of all small diagrams exist. 
\end{theorem}
\begin{proof}
   According to the proposition given in \S 2 of Chapter 5 in \cite{MacLane}, the statement of the theorem follows from the fact that the category $\mathsf{WC}_\mathsf{2}$ has all small products, all small coproducts, and equalizers and coequalizers of all pairs of parallel morphisms. 
\end{proof}

\section{Generalized Coxeter groups}

In this chapter we define generalized Coxeter groups, we describe some of their properties, and further we prove that the construction defines a functor from the category of weighted 2-complexes to the category of groups. 

\begin{definition}
Let $C = (V \sqcup E \sqcup F, \iota, \alpha, \omega, \partial, \mu)$ be a weighted 2-complex.
The \textit{generalized Coxeter group} is the group $G(C)$, defined by the generating set $V$ and the relations of the following three types:
\begin{align*}
    v^2 &= 1,\,& &\text{for any vertex}\ v \in V,\\
    (\alpha(e)\omega(e))^{\mu(e)} &= 1,\,& &\text{for any edge}\ e \in E,\\
    (\alpha(e_1) \ldots \alpha(e_t))^{\mu(f)} &= 1,\,& &\text{for any 2-cell}\ f \in F\ \text{with the boundary}\ \partial(f) = [e_1,\ldots,e_t].
\end{align*}
\end{definition}

Note that distinct ordered tuples of edges may define the boundary of the same 2-cell: if $f \in F$ and $\partial(f) = [e_1,e_2,\ldots,e_t]$, then
$$\partial(f) = [e_1,e_2,\ldots,e_t] = [e_t, e_1,\ldots,e_{t-1}] = [e_t, e_{t - 1}, \ldots, e_1].$$
In other words, cyclic shifts and inversion of the list of edges do not change the cycle which defines the boundary of the 2-cell. 

\begin{proposition}
The group $G(C)$ is well defined, i.e. it does not depend on the choice of the set of edges $e_1,\ldots,e_t$ defining the boundary $\partial(f) = [e_1,\ldots,e_t]$ for each 2-cell $f \in F$.
\end{proposition}
\begin{proof}
Let $f \in F$ be a 2-cell in $C$ with $\partial(f) = [e_1,\ldots,e_t]$.
For every $i = 1,\ldots,t$ denote $v_i := \alpha(e_i)$. 
First we show that the group $G(C)$ does not depend on the choice of orientation of the 2-cell $f$:
$$1 = (v_1 \ldots v_t)^{\mu(f)}\, \Leftrightarrow\, 1 = (v_1 v_2 \ldots v_{t-1} v_t)^{-\mu(f)} = (v_t^{-1} v_{t-1}^{-1} \ldots v_2^{-1} v_1^{-1})^{\mu(f)} = (v_t v_{t-1} \ldots v_2 v_1)^{\mu(f)}.$$
Now we show that $G(C)$ does not depend on the choice of the initial vertex in the 2-cell $f$:
$1 = (v_1 \ldots v_t)^{\mu(f)}\ \Longleftrightarrow\ v_t v_{t-1} \ldots v_2 v_1 = (v_1 v_2 \ldots v_{t-1} v_t)^{\mu(f) - 1},$
\begin{align*}
    \Longleftrightarrow\ 1 &=\ v_t\ \LaTeXunderbrace{(v_1 v_2 \ldots v_{t-1} v_t)\ \ldots\ (v_1 v_2 \ldots v_{t-1} v_t)}_{\mu(f) - 1}\ v_1 v_2 v_3 \ldots v_{t-2} v_{t-1},\\
    \Longleftrightarrow\ 1 &=\ \LaTeXunderbrace{(v_t v_1 v_2 \ldots v_{t-2} v_{t-1})\ \ldots\ (v_t v_1 v_2 \ldots v_{t-2} v_{t-1})}_{\mu(f)} = (v_t v_1 v_2 \ldots v_{t-2} v_{t-1})^{\mu(f)}.\\
\end{align*}
\end{proof}

\subsection{Examples of generalized Coxeter groups}

The two simplest examples of weighted 2-complexes $C_\varnothing$ and $C_\text{pt}$ correspond to the two simplest examples of generalized Coxeter groups: $G(C_\varnothing) = 1,\ G(C_\text{pt}) = \mathbb{Z}_2$.
For a natural number $r \in \mathbb{N}$ define $V_r := \{v_1,\ldots,v_r\}$.
Consider the weighted 2-complex $C$ with the set of vertices $V_r$, without edges and without 2-cells.
In this case
$$G(C) = \langle v_1 | v_1^2 = 1 \rangle * \ldots * \langle v_r | v_r^2 = 1 \rangle = \LaTeXunderbrace{\mathbb{Z}_2 * \ldots * \mathbb{Z}_2}_{r}.$$
By connecting each pair of distinct vertices in the existing 2-complex with edges of weight $2$, we obtain the weighted 2-complex of the form: ${C = (V_r \sqcup V_r^2 \setminus \text{diag}\left(V_r^2\right), \iota, \alpha, \omega, \partial, \mu)}$, $\iota(v_i,v_j) = (v_j,v_i)$, $\alpha(v_i,v_j) = v_i$, $\omega(v_i,v_j) = v_j$,
$$\mu(x) = \begin{cases}
  1,\,& \text{if}\ x \in V_r,\\
  2,\,& \text{if}\ x \in V_r^2 \setminus \text{diag}\left(V_r^2\right).
\end{cases}$$
The generalized Coxeter group which corresponds to the constructed 2-complex is the direct product of $r$ copies of $\mathbb{Z}_2$
$$G(C) = \LaTeXunderbrace{\mathbb{Z}_2 \times \ldots \times \mathbb{Z}_2}_{r}.$$
Now consider the following weighted 2-complex:
$$C = (\{v_1,\ldots,v_n\} \sqcup \{(v_i,v_{i+1}),(v_{i+1},v_i)\}_{i=1}^{n-1} \sqcup \{(v_i,v_j)\colon |i-j| > 1\}, \iota, \alpha, \omega, \partial, \mu),$$
where $\alpha(v_i,v_j) = v_i$, $\omega(v_i,v_j)=v_j$, $\mu \colon (v_{i+1},v_i) \mapsto 3$, $(v_i,v_{i+1}) \mapsto 3$, for all $i,j\in \{1,\ldots,n-1\}$, $\mu \colon (v_i,v_j) \mapsto 2$ with $|i-j|>1$.
Then
$$G(C) = \langle v_1,\ldots,v_n\, |\ v_i^2 = 1,\ v_i v_{i+1} v_i = v_{i+1} v_i v_{i+1},\ v_iv_j=v_jv_i\ \text{with}\ |i-j|>1\ \rangle \cong S_{n+1}$$ is the symmetric group of order $(n+1)!$ .
Consider another weighted 2-complex,
$$C = (\{u,v\},\, \{(u,v), (v,u)\},\ \varnothing,\ \alpha, \omega, \partial, \mu, \nu),$$
where $\alpha(u,v) = u = \omega(v,u)$, $\alpha(v,u) = v = \omega(u,v)$, $\mu \colon (s,t) \mapsto n$, $(t,s) \mapsto n$ for some $n \in \mathbb{N}$.
Then
$$G(C) = \langle u, v\, |\ u^2 = v^2 = (uv)^n = 1\ \rangle \cong D_n$$
is the dihedral group of order $2n$.

Recall that the Coxeter group (see \cite{Cox}) is defined as a group which admits the following presentation $W = \langle\, s_1,\ldots,s_n\ |\ (s_is_j)^{m_{ij}} = 1\, \rangle$, where $m_{ij} \in \mathbb{N} \cup \{\infty\}$, $m_{ij} \geqslant 2$ with $i \neq j$, $m_{ii} = 1$.
For a Coxeter group $W$, the pair $(W,S)$ is called a Coxeter system, where $S = \{s_1,\ldots,s_n\}$ is the generating set.

\begin{proposition}
Any Coxeter group $W$ together with a chosen Coxeter system $(W,S)$ is a generalized Coxeter group.
\end{proposition}
\begin{proof}
    Let $(W,S)$ be a Coxeter system, ${S = \{s_i\}_{i \in I}}$.
    Consider the weighted 2-complex
    $$C = (S,\ S^2 \setminus \text{diag}\left(S^2\right),\ \varnothing,\ \alpha,\ \omega,\ \partial,\ \mu,\ \nu),$$
    where $\alpha(s_i,s_j) = s_i$, $\omega(s_i,s_j) = s_j$, $\mu(s_i,s_j) = m_{ij}$, while $\partial$ and $\nu$ are empty functions.
    It is easy to see that $G(C) \cong W$.
\end{proof}

The two most important examples of generalized Coxeter groups are Gauss virtual braid groups on $n$ strands $GV\!P_n$ \cite{GVB} and the group of $k$-free braids on $n$ strands $G_n^k$ \cite{Man}.
The group $GV\!P_n$ admits the following presentation:
\begin{align*}
    GV\!P_n = \langle \lambda_{i,j}, 1 \le i < j \le n\, |\ &\lambda_{i,j}^2=1,\ \lambda_{i,j}\lambda_{k,l} = \lambda_{k,l}\lambda_{i,j},\ \lambda_{k,i}\lambda_{k,j}\lambda_{i,j} = \lambda_{i,j}\lambda_{k,j}\lambda_{k,i},\\
    &\text{where}\ i, j, k, l\ \text{denote pairwise distinct indices}\ \rangle.
\end{align*}
Let $n,k \in \mathbb{N}$.
We denote $[n] := \{1,\ldots,n\}$ and $[n]^k := \{\rho \subseteq [n] \colon |\rho| = k\}$ the set of all  $k$-element subsets in $[n]$.
The group $G_n^k$ is given by the generating set $[n]^k$ and relations of the following three types:
\begin{align*}
    \rho^2 = 1, & \,& &\text{for any}\ \rho \in [n]^k;\\
    \rho_1 \rho_2 = \rho_2 \rho_1, & \,& &\text{for}\ \rho_1,\rho_2 \in [n]^k\ \text{such that}\ |\rho_1 \cap \rho_2| < k-1;\\
    \rho_{\sigma(1)} \ldots \rho_{\sigma(k+1)} = \rho_{\sigma(k+1)} \ldots \rho_{\sigma(1)}, & \,& &\text{where}\ \rho_j := U \setminus \{u_j\},\ \text{for all}\ U = \{u_1,\ldots,u_{k+1}\}\in [n]^{k+1}\\
    & \,& &\text{and for all permutations}\ \sigma \in S_{k+1}.
\end{align*}
From the presentations of groups $GV\!P_n$ and $G_n^k$ given above, the existence and the explicit construction of weighted 2-complexes $C_n$ and $C_{n,k}$ such that $G(C_n) = GV\!P_n$ and $G(C_{n,k}) = G_n^k$ are obvious.

\end{document}